\documentclass[10.5pt]{article}
\usepackage{mathrsfs}
\usepackage{bbm}
\usepackage[leqno]{amsmath}
\usepackage{amsfonts}
\usepackage{graphicx}
\usepackage{amsmath}
\usepackage{amssymb}
\usepackage{latexsym}
\usepackage{amsmath, amsfonts,amssymb, amsthm, euscript,makeidx,color,mathrsfs}
\usepackage{indentfirst}

\def\sqr#1#2{{\vcenter{\vbox{\hrule height.#2pt
              \hbox{\vrule width.#2pt height#1pt \kern#1pt \vrule width.#2pt}
              \hrule height.#2pt}}}}

\def\3n{\negthinspace \negthinspace \negthinspace }
\def\2n{\negthinspace \negthinspace }
\def\1n{\negthinspace }

\def\dbC{\mathbb{C}}

\def\dbE{\mathbb{E}}
\def\dbF{\mathbb{F}}

\def\dbH{\mathbb{H}}

\def\dbN{\mathbb{N}}

\def\dbP{\mathbb{P}}

\def\dbR{\mathbb{R}}
\def\dbS{\mathbb{S}}

\def\sD{\mathscr{D}}

\def\={\buildrel \triangle \over =}

\def\ds{\displaystyle}

\def\ns{\noalign{\ss}}
\def\a{\alpha}
\def\b{\beta}

\def\d{\delta}

\def\si{\sigma}

\def\f{\varphi}

\def\o{\omega}

\def\i{\infty}
\def\Th{\Theta}

\def\F{\Phi}
\def\O{\Omega}

\def\cA{{\cal A}}

\def\cF{{\cal F}}

\def\cL{{\cal L}}

\def\cS{{\cal S}}

\def\cU{{\cal U}}

\def\no{\noindent}

\def\ss{\smallskip}
\def\ms{\medskip}

\def\q{\quad}
\def\qq{\qquad}

\def\limsup{\mathop{\overline{\rm lim}}}
\def\liminf{\mathop{\underline{\rm lim}}}

\def\lan{\mathop{\langle}}
\def\ran{\mathop{\rangle}}

\def\argmin{\mathop{\rm argmin}}

\def\essinf{\mathop{\rm essinf}}

\def\cd{\cdot}

\def\ae{\hbox{\rm a.e.{ }}}

\def\tr{\hbox{\rm tr$\,$}}

\def\les{\leqslant}
\def\ges{\geqslant}

\def\({\Big (}
\def\){\Big )}
\def\[{\Big[}
\def\]{\Big]}
\def\bde{\begin{definition}}
\def\ede{\end{definition}}
\def\be{\begin{equation}}
\def\bel{\begin{equation}\label}
\def\ee{\end{equation}}
\def\bt{\begin{theorem}}
\def\et{\end{theorem}}
\def\bc{\begin{corollary}}
\def\ec{\end{corollary}}
\def\bl{\begin{lemma}}
\def\el{\end{lemma}}
\def\bp{\begin{proposition}}
\def\ep{\end{proposition}}
\def\bas{\begin{assumption}}
\def\eas{\end{assumption}}
\def\br{\begin{remark}}
\def\er{\end{remark}}
\def\ba{\begin{array}}
\def\ea{\end{array}}
\def\ed{\end{document}}

\def\square#1{\vbox{\hrule\hbox{\vrule height#1%
     \kern#1\vrule}\hrule}}
\def\rectangle#1#2{\vbox{\hrule\hbox{\vrule height#1%
     \kern#2\vrule}\hrule}}

\font\tenbb=msbm10 \font\sevenbb=msbm7 \font\fivebb=msbm5

\newfam\bbfam
\scriptscriptfont\bbfam=\fivebb \textfont\bbfam=\tenbb
\scriptfont\bbfam=\sevenbb

\newtheorem{lemma}{Lemma}[section]
\newtheorem{remark}{Remark}[section]
\newtheorem{example}{Example}[section]
\newtheorem{theorem}{Theorem}[section]
\newtheorem{corollary}{Corollary}[section]

\newtheorem{definition}{Definition}[section]
\newtheorem{proposition}{Proposition}[section]
\newtheorem{assumption}{Assumption}[section]

\makeatletter
   
   \@addtoreset{equation}{section}
\makeatother

\begin{document}
\title{\bf Stochastic Verification Theorems  for Stochastic Control Problems of Reflected FBSDEs\thanks{This work is supported  by NSF of
P.R.China (No. 11971099) and  NSF of Jilin Province for Outstanding Young Talents (No. 20230101365JC).
  }}

\author{Lu Liu\footnote{School of Mathematics and Statistics, Northeast Normal University, Changchun 130024, China; email: {\tt liulu@nenu.edu.cn}}
 \q\qq Xinlei Hu\footnote{School of Mathematics and Statistics, Northeast Normal University, Changchun 130024, China; email: {\tt huxl302@nenu.edu.cn}}
 \q\qq Qingmeng Wei\footnote{Corresponding author. School of Mathematics and Statistics, Northeast Normal University, Changchun 130024, China; email: {\tt weiqm100@nenu.edu.cn}}
 }

\maketitle

\begin{abstract}
In this paper,   the stochastic verification theorems for stochastic control problems of reflected forward-backward stochastic differential equations are studied.  We carry out the work  within the frameworks of   classical  and viscosity solutions.
The sufficient conditions of verifying the controls to be optimal are given.
We also construct the feedback optimal control laws  from the classical and viscosity solutions of the associated Hamilton-Jacobi-Bellman  equations with obstacles.
Finally, we apply the theoretical results in two concrete examples.
One is for the case   of   the classical solution, and the other  is for the case of the viscosity solution.
\end{abstract}
\ms


\section{Introduction}\label{Sec_In}

\par   Dynamic programming method  originated by Bellman in the early 1950s, is one of the   powerful tools to solve the optimal control problems. The main idea of the method is to study a family of optimal control problems with different initial times and states, and to establish the relationship among them by the associated  Hamilton-Jacobi-Bellman (HJB, for short) equations.   It  has been comprehensively applied to the deterministic and stochastic control problems, referring to Fleming, Rishel \cite{FR}, Yong, Zhou \cite{YZ} and the references therein.
With  the development of nonlinear  backward stochastic differential equations (BSDEs, for short) firstly introduced by Pardoux, Peng \cite{PP-1990},   a series of relevant stochastic control problems  spring up. While not exceptional, the method of dynamic programming principle (DPP, for short) grows rapidly and has been applied widely to these control problems, 
referring  to \cite{BL-2009,  LW,  Peng-1992, Peng-1997, WY-2008}, etc. The value functions of  these different stochastic control problems were shown to be  the solutions of the corresponding partial differential equations (PDEs, for short).
 However,
most of references omit the construction of optimal controls from the solutions of PDEs, which is actually the task of stochastic verification theorem.
Stochastic verification theorems providing the sufficient conditions of verifying the controls to be optimal by virtue of  the solutions of PDEs, are important and  indispensable, especially in engineering supervision, numerical calculations and algorithm designs.


 It is relatively easy to get the verification theorems when the solutions of PDEs are classical, i.e., smooth enough, referring to \cite{FR, YZ} for the control problems and \cite{WY} for the game problems. However, it is difficult for PDEs  to have the classical solutions, so that we have to resort to the  weak solutions.
As we see in the  references  about DPP  mentioned before, most of the frameworks  involve a kind of weak solution, i.e., the viscosity solution, which was   introduced firstly for the first-order Hamilton-Jacobi equations by Crandall, Lions \cite{CL}, and then
developed for the second order PDEs by Crandall et al.   \cite{CIL}.
 Under  the framework of viscosity solutions, Zhou et al. \cite{ZYL}, Gozzi et al. \cite{GSZ1, GSZ2} studied the stochastic verification theorem for the classical   stochastic control problem;
 Zhang \cite{Zhang} got the stochastic verification theorem for the stochastic recursive  control problem.
 Recently, Chen, L\"{u} \cite{CL-arxiv} established the stochastic verification theorem for infinite
dimensional stochastic control systems directly from   DPP of the value function without the enough smoothness.

 This work shall focus on a type of stochastic control problems of reflected forward-backward stochastic differential equations (FBSDEs, for short).
Reflected BSDE was first introduced by El Karoui et al. \cite{El-KPPQ} and developed widely in many aspects, including  DPP.
By establishing the DPP of the control or game problems,  the references \cite{BL-2009, BL-2011, WY-2008}  showed that the value functions were the viscosity solutions of the associated HJB or HJBI equations with obstacles.
In this paper, we aim to study the verification theorems of stochastic control problems of reflected FBSDEs, and construct the feedback optimal control laws from the HJB equations with obstacles.
The research will be  carried out within  the frameworks of   classical solutions and viscosity solutions.

Firstly, we  present the sufficient conditions  of the  controls to be optimal when the HJB equation with obstacle has the classical solution.
  The comparison between the BSDE which $W\big(\cd,X^{t,x;u}(\cd)\big)$  satisfies  and the  reflected BSDE of cost functional fails.
 Therefore, we convert to the comparison theorem of BSDEs by employing  the penalization sequence of the  reflected BSDEs, and further complete the proof.

For the viscosity solutions case, due to the lack of   enough smoothness,
the notions  of  second-order parabolic superdifferentials and subdifferentials are adopted to provide the smooth test functions.
Different from the    classical solutions case, some additional properties of  viscosity solutions of HJB equations with obstacles  are necessary.
The first  is the joint Lipschitz  continuous property of the viscosity solutions in $(t,x)$, the second is
  the semiconcavity  of the viscosity solutions with respect to $x$.
It is a bit  restrictive. However, it is fortunately that the   study in \cite{BHL2012}  makes  the two properties  be possible.
 Further,  two examples are presented to illustrate that  the obtained verification theorems give a way to construct an optimal control or to test whether a given admissible control  is optimal.

The  structure of our manuscript is as follows. We   formulate the control problem  and recall some known results    in Section  2.   Section 3   is about stochastic verification theorem within the framework of classical solutions.
In Section 4, we carry out the detailed  study  for the case of viscosity solutions and construct the feedback optimal control laws.
 Finally, the theoretical results are applied in  two calculable examples.
  One is for the classical solutions case,  the other is for the viscosity solutions case.

\section{Preliminaries}\label{Sec_Pre}

\par Let $(\Omega,\mathcal F,\mathbb F,\mathbb P)$ be a complete filtered probability space on which a $d$-dimensional standard Brownian
motion $B(\cdot)$ is defined, and $\mathbb F = \{ \mathcal F_t
\}_{t\ges 0}$ is its natural filtration augmented by all the
$\mathbb P$-null sets. Let $T>0$ be a given terminal time.
For any
$t\in [0,T]$, $k\ges 2$,  and  Euclidean space
$\mathbb R^n$ ($n\ges 1$), we introduce the following  spaces,
$$
\ba{ll}
\ns\ds L^k_{\cF_t}( \O;\mathbb R^n):=  \Big\{ \xi:
\Omega\rightarrow\mathbb R^n  \mid \xi \mbox{ is
 } \cF_t\mbox{-measurable,}\  \mathbb
E|\xi|^k<\infty \Big\};\qq \sD=[0,T]\times L_{\cF_t}^2(\O;\dbR^n);\\
\ns\ds\cS^k_{\mathbb F}( t,T;\mathbb R^n):=  \Big\{ \f:
\Omega\times [t,T]\rightarrow\mathbb R^n  \mid \f(\cdot) \mbox{ is
 } \mathbb F\mbox{-adapted,  continuous, and}\  \mathbb
E\Big[\sup_{r\in
[t,T]} |\f(r)|^k\Big]<\infty \Big\};\qq \qq\qq\qq\\
\ea
$$
$$
\ba{ll}
\ns\ds L^k_{\mathbb F}(t,T;\mathbb R^n) :=  \Big\{ \f:\Omega\times
[t,T]\rightarrow\mathbb R^n \mid  \f(\cdot) \mbox{ is } \mathbb
F\mbox{-progressively measurable,}\  \mbox{and  } \mathbb E\(\int_t^T |\f(r)|^2 dr\)^\frac k2
<\infty \Big\};\\
\ns\ds\cA_c^2(t,T;\dbR) :=  \Big\{ \f:\Omega\times
[t,T]\rightarrow\mathbb R\mid \f(\cdot) \mbox{ is } \mathbb
F\mbox{-adapted, continuous and increasing, }\f(t)=0,\ \dbE|\f(T)|^2<\i  \Big\};\\
\ns\ds C^{1,2}([t,T]\times\mathbb R^n):=   \Big\{ w: [t,T]\times\mathbb
R^n\rightarrow\mathbb R\mid   w(\cdot,\cdot) \mbox{ is
continuous, } w_r(\cdot,\cdot), w_x(\cdot,\cdot), w_{xx}(\cdot,\cdot) \mbox{ exist and  }\\
\ns\ds \hskip 6.6cm
\mbox{    are also continuous}\Big\}.
\ea
$$

\ms

Now we formulate the stochastic control problem.  For any $t\in[0,T]$, denote $\cU_{t,T}$  by the set of all the admissible controls on $[t,T]$, that is,
$$\ \cU_{t,T} := \Big\{ u :[t,T]\times\Omega \rightarrow U\mid  u (\cdot) \mbox{ is } \mathbb F\mbox{-progressively
measurable}\Big\},
 $$
where  $U\subseteq\dbR^m$ is the  nonempty compact  set.

For any $(t,\xi)\in\sD$ and $u(\cd)\in\cU_{t,T}$, consider the following controlled stochastic differential equation (SDE, for short),
\begin{equation}\label{SDE}
\left\{
\ba{ll}
\ns \ds\!\!\! dX(s)  = b\big(s,X(s) ,u(s)   \big) ds +\sigma\big(
s,X(s) ,u(s)    \big) dB(s)  ,\quad s\in [t,T],\\
\ns \ds\!\!\! X(t)  =\xi ,
\ea
\right.
\end{equation}
where  $u(\cd)$ is the control process, $X(\cd)$ is the controlled state process and the
coefficients
 $$\ba{ll}
 b:[0,T]\times\mathbb R^n\times U  \rightarrow
\mathbb R^n,\q  \sigma: [0,T]\times\mathbb R^n\times U  \rightarrow\mathbb R^{n\times d}, \
\ea$$
satisfy

\noindent{\bf (H1).} (i). for every fixed $ x  \in \dbR^n $, $b(\cd,x,\cd)$, $\si(\cd,x,\cd)$  are continuous in $(r,u)\in[0,T]\times U$;

 \hskip 0.58cm (ii). there exists some constant $C>0$ such that, for any
 $(r,u)\in [0,T]\times U$,   $x,x'\in\dbR^n$,
$$  |b(r,x,u)-b(r,x',u)|+ |\si(r,x,u)-\si(r,x',u)|\les C|x-x'|.
$$
 Obviously, under {\bf (H1)},
for any $k\ges 2$,  $\xi\in L^k_{\mathcal F_t}(\Omega;\mathbb R^n)$,  \eqref{SDE} admits the unique $\dbF$-adpated solution $X(\cd)\equiv X^{t,\xi;u}(\cd)\in \cS^k_{\mathbb F}(t,T;\mathbb R^n)$.
And for any $t\in[0,T]$,  $\xi,\xi'\in L^k_{\mathcal F_t}(\Omega;\mathbb R^n)$, $0\les h\les T-t$,  there exist  some constants $C>0$ such that the following estimates hold, $P$-a.s.,
\bel{est888}\ba{ll}
\ns\ds\dbE_t \[\sup\limits_{r\in[t,T]}|X^{t,\xi;u}(r) |^k\] \les C(1+|\xi|^k),\\
\ns\ds\dbE_t \[\sup\limits_{r\in[t,t+h]}|X^{t,\xi;u}(r) -\xi|^k\] \les Ch(1+|\xi|^k),\\
\ns\ds \dbE_t\[\sup\limits_{r\in[t,T]}|X^{t,\xi;u}(r) -X^{t,\xi';u}(r) |^k\]\les C  |\xi-\xi'|^k,
\ea\ee
where $\dbE_t[\cd]:=\dbE[\cd\mid\cF_t]$ for simplicity. The  details of \eqref{est888} can be referred to \cite{YZ}.

For any $(t,x)\in[0,T]\times\dbR^n$, we call  $(u(\cd), X^{t,x;u}(\cd))\in\cU_{t,T}\times  \cS^2_{\mathbb F}(t,T;\mathbb R^n)$  the admissible pair, $X^{t,x;u}(\cd)$ an admissible state process.

 \ms

Next, to introduce the  cost functional, we consider  the following reflected BSDE, for any $(t,\xi)\in\sD$ and $u(\cd)\in\cU_{t,T}$,
\bel{RBSDE}
\left\{\ba{ll}
 \ns\ds \!\!\!\! {\rm(i).} \ \! (Y(\cd),Z(\cd),K(\cd))\in\cS_\dbF^2(t,T;\dbR)\times L_\dbF^2(t,T;\dbR^d)\times \cA_c^2(t,T;\dbR);\\
 \ns\ds\!\!\!\!  {\rm(ii).} \ \! Y(s)\!=\!\F\big(X(T)\big)\!+\!\int_s^T  \!\!\! f\big(r,X(r),Y(r),Z(r),u(r)\big)dr\!-\!\big(K(T)\!-\!K(s)\big)\!-\!\int_s^T\!\!\!Z(r)dB(r),\ \! s\in [t,T];\\
 \ns\ds\!\!\!\! {\rm(iii).} \ \! Y(s)\les h\big(s,X(s)\big), \mbox{ a.e. } s\in[t,T];\\
 \ns\ds\!\!\!\! {\rm(iv).} \ \! \int_t^T\(h\big(s,X(s)\big)-Y(s)\)dK(s)=0 ,
\ea\right.\ee
where $X(\cd)$ satisfies  \eqref{SDE}, the driver $f:[0,T]\times\dbR^n\times \dbR\times\dbR^d\times U\to\dbR$, the terminal condition $\F:\dbR^n\to\dbR$ and the obstacle term $h:[0,T]\times\dbR^n\to\dbR$ are assumed to satisfy
\ss

\noindent{\bf (H2).} (i). for every $(x,y,z )\in\dbR^n\times \dbR\times \dbR^d $,  $f(\cd,x,y,z, \cd)$ is continuous in $(r,u)\in[0,T]\times\dbR^n$, $h(\cd,x)$ is continuous in $t\in[0,T]$;

 \hskip 0.71 cm (ii). there exist some constants $C>0$ such that,  for any
 $(r,u)\in [0,T]\times
U$, $x$, $x'\in\dbR^n$, $y$, $y'\in\dbR$, $z$, $z'\in\dbR^d$,
$$\ba{ll}
\ns\ds |f(r,x,y,z,u)-f(r,x',y',z',u)|\les C\big(|x-x'|+|y-y'|+|z-z'|\big),\\
\ns\ds |\Phi(x)-\Phi(x')|\les C|x-x'|,\qq  |h(r,x)-h(r,x')|\les C|x-x'|;
\ea$$

\hskip 0.71cm   (iii).  for any $x\in\dbR^n$, $ \F(x)\les h(T,x)$.
\ms

According to the theory of reflected BSDEs (\cite{El-KPPQ, BL-2011}),   {\bf (H1)} and {\bf (H2)} guarantee that, for any $(t,\xi)\in\sD$, $u(\cd)\in\cU_{t,T}$,  there exists  the unique triple of   $(Y(\cd),Z(\cd),K(\cd))\equiv(Y^{t,\xi;u}(\cd),Z^{t,\xi;u}(\cd),K^{t,\xi;u}(\cd))$ satisfying \eqref{RBSDE}.
Now, for any initial pair $(t,x)\in[0,T]\times\dbR^n$ and the admissble control $u(\cd)\in\cU_{t,T}$, we can define
\bel{cost}
J(t,x;u(\cd)):=Y^{t,x;u}(t),
\ee
 which is the cost functional of our control problem. Note that, for $(t,\xi)\in\sD$ and $u(\cd)\in\cU_{t,T}$, we also have
$
J(t,\xi;u(\cd))=J(t,x;u(\cd)) |_{x=\xi} =Y^{t,\xi;u}(t), \  \dbP\mbox{-a.s.}, $  referring to \cite{BL-2009, BL-2011, WY-2008}.

\ms

 Based on the above preparation, now  we can formulate the control problem as follows,

\no {\bf Problem (RC)}  {\sl For any $x\in\dbR^n$, find $\bar u(\cd)\in\cU_{0,T}$ such that
  \bel{OPC-0}
  J(0,x;\bar u(\cd))=\essinf_{u(\cd)\in \cU_{0,T}}J(0,x;u(\cd)).
  \ee
  $\bar u(\cd)$ satisfying \eqref{OPC-0} is said to be the optimal control of Problem (RC),  the corresponding $\bar X(\cd)=X^{0,x;\bar u}(\cd)$ is the optimal state process.
}

\ss

To get the optimal control, we need to study the following family of control problems parameterized by the different initial pairs $(t,x)\in[0,T]\times\dbR^n$.
\ss

\no {\bf Problem (RC)$_{t,x}$}  {\sl For any $(t,x)\in[0,T]\times\dbR^n$, find $\bar u(\cd)\in\cU_{t,T}$ such that
  \bel{OPC}
  J(t,x;\bar u(\cd))=\essinf_{u(\cd)\in \cU_{t,T}}J(t,x;u(\cd)):=V(t,x).
  \ee
 In the above, $\bar u(\cd)$ satisfying \eqref{OPC} is called  as the optimal control of Problem (RC)$_{t,x}$,  the corresponding $\bar X(\cd)=X^{t,x;\bar u}(\cd)$ is the optimal state process. We call $(\bar X(\cd),\bar u(\cd)) $ as the optimal pair of Problem (RC)$_{t,x}$, and $V:[0,T]\times\dbR^n\to\dbR$ as  the value function of  Problem (RC)$_{t,x}$.
}

\ss

\ss

By \cite{BL-2009, BL-2011, WY-2008}, the value function $V(\cd,\cd)$ possesses the following properties.

\bl\label{Le-con}\sl Under  {\bf (H1)} and {\bf (H2)}, $V(\cd,\cd)$ is  deterministic,
 Lipschitz continuous and linear growth  in $x\in\dbR^n$, and   $\frac 12$-H\"{o}lder continuous in  $t\in[0,T]$.

\el

In addition, the value function $V(\cd,\cd)$  can solve   the  following HJB equation  with obstacle   in some sense,
%
%
 \bel{HJB}
\left\{
\ba{ll}
\ns\ds\!\!\!\! \max\Big\{ \! W(t,x)-h(t,x),- \frac{\partial}{\partial t}W(t,x) \!-\!\inf_{u\in U}  \mathbb H \big(
t,x,W(t,x),W_x(t,x), W_{xx}(t,x),u \big)  \Big\}=0,\\
\ns\ds\hskip 10.5cm
(t,x)\in[0,T]\times\dbR^n,\\
\ns\ds\!\!\!\! W(T,x) = \Phi(x),\quad x\in \mathbb R^n,
\ea
\right.
\ee
where
\bel{Hami}\left\{\ba{ll}
\ns\ds\!\!\! \mathbb H(r,x, y,p,P,u) :=  \tr
[a(r,x,u)P] +p_\cdot b(r,x,u)  +f\big(r,x,y,p_\cdot\si(r,x,u), u \big),\\
\ns\ds\!\!\! a(r,x,u):= \frac 1 2
\sigma(r,x, u)\sigma(r,x,u)^\top,\q     (r,x,y,p,P,u)\in
[0,T]\times \dbR^n\times\dbR\times \dbR^n\times \dbS^n\times U. \ea\right.\ee
Here  $\dbS^n$ is the set of all the  $n\times n $ symmetric matrices.
%
%
 Precisely,

\bl\label{Le-vis}\sl Under  {\bf (H1)} and {\bf (H2)}, $V(\cd,\cd)$  is  the unique viscosity solution (unique in $C_p([0,T]\times\dbR^n)$) of the HJB equation  with obstacle \eqref{HJB}, where $C_p([0,T]\times\dbR^n)$ is the space of continuous real functions over $[0,T]\times\dbR^n$ which have polynomial growth.

\el

The above two results are classical, so that we will not repeat the details, which  can be referred to \cite{BHL2012, WY-2008}, including the definition of viscosity solutions.

Before ending this section,  we introduce the   following definition of admissible feedback control laws which will be needed later.

\begin{definition}\sl   A measurable mapping $\mathbbm{u}:[t,T]\times\mathbb R^n \rightarrow U$ is called  an {\it admissible  feedback control law}, if for any
	$(t,x)\in[0,T]\times  \mathbb R^n $,    the following
	\bel{SDE-u}
	\left\{
	\ba{ll}
	\ns\ds\!\!\!   dX^{t,x;\mathbbm{u}}(s)  = b\big ( r,X^{t,x;\mathbbm{u}}(r) ,\mathbbm{u}(r,
	X^{t,x;\mathbbm{u}}(r) )   \big)dr + \sigma\big( r,X^{t,x;\mathbbm{u}}(r) ,\mathbbm{u}(r,X^{t,x;\mathbbm{u}}(r))   \big)dB(r)  , \ s\in[t,T],\\
	\ns\ds\!\!\!  X^{t,x;\mathbbm{u}}(t)  = x,
	\ea
	\right.
	\ee
	and
	\bel{RBSDE-u}
	\left\{\ba{ll}
	\ns\ds \!\!\! {\rm(i).} \  (Y^{t,x;\mathbbm{u}}(\cd),Z^{t,x;\mathbbm{u}}(\cd),K^{t,x;\mathbbm{u}}(\cd))\in\cS_\dbF^2(t,T;\dbR)\times L_\dbF^2(t,T;\dbR^d)\times \cA_c^2(t,T;\dbR);\\
	\ns\ds\!\!\!  {\rm(ii).} \ Y^{t,x;\mathbbm{u}}(s)=\F\big(X^{t,x;\mathbbm{u}}(T)\big)\!+\!\int_s^T f\big(r,X^{t,x;\mathbbm{u}}(r),Y^{t,x;\mathbbm{u}}(r),Z^{t,x;\mathbbm{u}}(r),\mathbbm{u}(r,X^{t,x;\mathbbm{u}}(r)) \big)dr\\
	\ns\ds\qq\qq\qq\q \!-\! \big(K^{t,x;\mathbbm{u}}(T)-K^{t,x;\mathbbm{u}}(s)\big)\!-\!\int_s^TZ^{t,x;\mathbbm{u}}_rdB(r),\q  s\in [t,T];\\
	\ns\ds\!\!\! {\rm(iii).} \ Y^{t,x;\mathbbm{u}}(s)\les h\big(s,X^{t,x;\mathbbm{u}}(s)\big), \mbox{ a.e. } s\in[t,T];\\
	\ns\ds\!\!\! {\rm(iv).} \ \int_t^T\Big(h\big(r,X^{t,x;\mathbbm{u}}(r)\big)-Y^{t,x;\mathbbm{u}}(r)\Big)dK^{t,x;\mathbbm{u}}(r)=0,
	\ea\right.\ee
	admit  the unique adapted  solutions $ X^{t,x;\mathbbm{u}}(\cd)$ and $(Y^{t,x;\mathbbm{u}}(\cd),Z^{t,x;\mathbbm{u}}(\cd),K^{t,x;\mathbbm{u}}(\cd))$, respectively.

\end{definition}

Note that, the outcome  $\mathbbm{u}(\cdot,X^{t,x;\mathbbm{u}}(\cd))$ of the admissible  feedback control law  $\mathbbm{u}(\cd,\cd)$ is still our admissible control, i.e.,
$\mathbbm{u}(\cdot,X^{t,x;\mathbbm{u}}(\cd)) \in
\cU_{t,T}$.

\section{Stochastic Verification Theorem: Classical Solutions}\label{Sec_SVP-C}
 \par  In the section, we try to construct the optimal control of Problem (RC)$_{t,x}$  from the classical solution of HJB equation  with obstacle \eqref{HJB}.
 For this, we strengthen  the continuity conditions of the coefficients $b$, $\si$ and  $f$ on the control variable $u$ as follows,

\ss

 \no{\bf(H3).} for every  $ (r,x,y,z) \in [0,T]\times\dbR^n\times \dbR \times \dbR^d $, $b(r,x,\cd)$, $\si(r,x,\cd)$, $f(r,x,y,z,\cd)$ are Lipschitz continuous in $u\in U $.

\ms
Further, denote $\mathcal{L}$ by the class of measurable mappings $\mathbbm{u}:[0,T]\times\mathbb R^n \rightarrow U$ with the following properties,
%
 $$\left\{\ba{ll}
 \ns\ds\!\!\!\!\!  \mbox{ (i). for every fixed } x  \in \dbR^n, u(\cd,x) \mbox{ is continuous in } r\in[0,T];\\

 \ns\ds\!\!\!\!\!  \mbox{ (ii). there exists some constant } C>0 \mbox{  such that, for any } r\in [0,T],   x,x'\in\dbR^n, \\
 \ns\ds \hskip0.6cm  |u(r,x)-u(r,x')|\les C|x-x'|.
\ea\right.
$$
 Note that, under $\mathbf{(H1)}$-$\mathbf{(H3)}$,  $\mathbbm{u}(\cd,\cd)\in\mathcal{L}$ is the admissible  feedback control law.p
For any  $(r,x,y,p,P)\in [0,T]\times \dbR^n\times\dbR\times\dbR^n\times\dbS^n$, we introduce the  mapping    $\psi:[0,T]\times \dbR^n\times\dbR\times\dbR^n\times\dbS^n\to U$  such that
 $$\ba{ll}
 \ns\ds \psi(r,x,y,p,P)\in\argmin\dbH(r,x,y,p,P,\cd)\equiv\Big\{\bar u\in U \mid \dbH(r,x,y,p,P,\bar u)=\min_{u\in U}\dbH(r,x,y,p,P,u)\Big\} .
 \ea$$

Now we present the  first main result.

	 \begin{theorem}\label{SVT-class}\sl
	 	Assume $\mathbf{(H1)}$-$\mathbf{(H3)}$. Let $W( \cdot , \cdot ) \in {C^{1,2}}([0,T] \times {\mathbb R^n})$ be the classical solution of the HJB equation with obstacle \eqref{HJB}.
 	Then

 {\rm(i).}
  	for any $(t,x) \in [0,T] \times {\mathbb R^n}$ and $u(\cd) \in \cU_{t,T}$, we have
 		$ W(t,x) \les J(t,x;u(\cd));$

 {\rm (ii).} for any $(t,x)\in[0,T]\times\dbR^n$, defining $\bar{\mathbbm{u}}:[t,T]\times\dbR^n\to U  $ as
 \bel{OC-C}
  \bar{\mathbbm{u}}(s,y)=\psi(s,y,W(s,y),W_{x}(s,y),W_{xx}(s,y)),\q (s,y)\in[t,T]\times\dbR^n,
 \ee
%
%
if $\bar{\mathbbm{u}} (\cd,\cd)\in\cL$, then $\bar{\mathbbm{u}}\big(\cd,X^{t,x;\bar{\mathbbm{u}}}(\cd)\big)$ is the optimal control of Problem (RC)$_{t,x}$, where $X^{t,x;\bar{\mathbbm{u}}}(\cd)$ satisfies
\eqref{SDE-u} with $\bar {\mathbbm{u}}(\cd,\cd).$
In this case,  $W(\cd,\cd)$ is indeed the value function $V(\cd,\cd)$, i.e.,
  $$W(t,x)=J\big(t,x;\bar{\mathbbm{u}}(\cd,X^{t,x;\bar{\mathbbm{u}}}(\cd))\big)=V(t,x),\q (t,x)\in[0,T]\times\dbR^n.$$

\end{theorem}

 	 \begin{proof}
 	 	{\rm (i).}
For any $(t,x) \in [0,T] \times {\mathbb R^n},$ $u(\cd)\in \cU_{t,T}$,	applying It\^o's formula to $W(\cd,X^{t,x;u}(\cd))$ and using \eqref{Hami}, we get
 	 	\bel{BSDE-W}\ba{ll}
       \ns\ds  W\big(s,X^{t,x;u}(s)\big)
      = \Phi \big(X^{t,x;u}(T)\big) - \int_s^T  W_x\big(r,X^{t,x;u}(r)\big). \sigma\big (r,X^{t,x;u}(r),u(r)\big)  dB(r)\\
       \ns\ds  \!+\! \int_s^T\! \!\[f\big(r,X^{t,x;u}(r),W\big(r,X^{t,x;u}(r)\big),W_x\big(r,X^{t,x;u}(r)\big). \sigma (r,X^{t,x;u}(r),u(r)), u(r) \big)\!- \!\frac{\partial }{{\partial r}}W\big(r,X^{t,x;u}(r)\big) \!\!\!\!\!\\
       \ns\ds \hskip1.1cm - \mathbb H\big(r,X^{t,x;u}(r),W\big(r,X^{t,x;u}(r)\big),W_x\big(r,X^{t,x;u}(r)\big),W_{xx}(r,X^{t,x;u}(r)\big),u(r)\big)\]dr,\q s\in [t,T].
 	 	\ea\ee
 %
%
%
 On the other hand, for any  $ (t,x) \in [0,T] \times \mathbb R^n$, $n \in \mathbf N^*$, we consider the following BSDEs,
\bel{Yn}\ba{ll}
 \ns\ds  {}^nY^{t,x;u}(s) = \Phi \big(X^{t,x;u}(T)\big) + \int_s^T f\big(r, X^{t,x;u}(r),{}^nY^{t,x;u}(r),{}^nZ^{t,x;u}(r),u(r)\big)dr\\
\ns\ds \hskip2cm - n\int_s^T \Big( {}^nY^{t,x;u}(r) - h(r,X^{t,x;u}(r)\big)\Big)^ + dr - \int_s^T {{}^nZ^{t,x;u}(r)}  dB(r),\quad s\in[t,T].
\ea\ee
By the fact that $W(\cd,\cd)$ being the classical solution  of \eqref{HJB}, we get the following two cases,

 {\rm Case (a).} at any point $(t,x)\in [0,T]\times\dbR^n$ where $(W-h)(t,x)= 0$,
 $$- \frac{\partial}{\partial t}W(t,x) -\inf_{u\in U}  \mathbb H \big(
t,x,W(t,x),W_x(t,x), W_{xx}(t,x),u \big) \les 0;$$

  {\rm Case (b).} at any point $(t,x)\in [0,T]\times\dbR^n$ where $(W-h)(t,x)< 0$,
 $$- \frac{\partial}{\partial t}W(t,x) -\inf_{u\in U}  \mathbb H \big(
t,x,W(t,x),W_x(t,x), W_{xx}(t,x),u \big) = 0.$$
No matter (a)  or  (b),   for any $r\in[t,T]$ and $u(\cd)\in \cU_{t,T}$, $n \in \mathbf N^*$, we have
 \bel{compar-f}\ba{ll}
 \ns\ds
 f\big(r,X^{t,x;u}(r),W\big(r,X^{t,x;u}(r)\big),W_x\big(r,X^{t,x;u}(r)\big). \sigma (r,X^{t,x;u}(r),u(r)),u(r)\big) - \frac{\partial }{{\partial r}}W\big(r,X^{t,x;u}(r)\big)\\
\ns\ds
 - \mathbb H\big(r,X^{t,x;u}(r),W\big(r,X^{t,x;u}(r)\big),W_x\big(r,X^{t,x;u}(r)\big),W_{xx}(r,X^{t,x;u}(r)\big),u(r)\big)\\
\ns\ds  \les f\big(r,X^{t,x;u}(r),W\big(r,X^{t,x;u}(r)\big),W_x\big(r,X^{t,x;u}(r)\big).\sigma (r,X^{t,x;u}(r),u(r)),u(r)\big)- \frac{\partial }{{\partial r}}W\big(r,X^{t,x;u}(r)\big)\\
\ns\ds\q
 	- \mathop {\inf }\limits_{u \in U}\mathbb H\big(r,X^{t,x;u}(r),W\big(r,X^{t,x;u}(r)\big),W_x\big(r,X^{t,x;u}(r)\big),W_{xx}(r,X^{t,x;u}(r)\big),u\big)\\
%
 %
 \ns\ds \les f\big(r,X^{t,x;u}(r),W\big(r,X^{t,x;u}(r)\big),W_x\big(r,X^{t,x;u}(r)\big) . \sigma (r,X^{t,x;u}(r),u(r)),u(r)\big)\\
 \ns\ds\q -n\(W\big(r,X^{t,x;u}(r)\big) - h\big(r,X^{t,x;u}(r)\big)\)^ + .
\ea\ee
Therefore, by using the comparison theorem of BSDEs to \eqref{BSDE-W} and \eqref{Yn}, for  all $n \in {\dbN^*}$,  we get
\bel{W-Yn}W \big(s,X^{t,x;u}(s)\big) \les {}^n{Y^{t,x;u}}(s),\quad  s \in [t,T],\q \dbP\mbox{-a.s.}\ee

Further, according to the penalized method proving the wellposedness  of reflected BSDEs (referring to \cite{El-KPPQ, BL-2009}),   we know, as $n\to\i$,
$ {}^n{Y^{t,x;u}}(s)\downarrow  Y^{t,x;u}(s),\  s \in [t,T],\ \dbP\mbox{-a.s.}$, where   $Y^{t,x;u}(\cd)$ is the first component of the solution of reflected BSDE \eqref{RBSDE}.
So, by letting $n\to \i$ in \eqref{W-Yn},	$$W\big(s,X^{t,x;u}(s)\big) \les Y^{t,x;u}(s),\q s \in [t,T], \q \dbP\mbox{-a.s.}$$
Especially,  when $s=t$,
 	\bel{Step1}
 	 		W(t,x) \les {Y^{t,x;u}}(t) = J\big(t,x;u(\cd)\big),\quad\mathrm{for\ any}\ u(\cd) \in \cU_{t,T}.\ee

 \ms
 {\rm(ii).} Let $X^{t,x;\bar{\mathbbm{u}}}(\cd)$ and $(Y^{t,x;\bar{\mathbbm{u}}}(\cd),Z^{t,x;\bar{\mathbbm{u}}}(\cd),K^{t,x;\bar{\mathbbm{u}}}(\cd))$ be the  solutions  of SDE \eqref{SDE-u} and reflected BSDE \eqref{RBSDE-u} with $\mathbbm{u}(\cd,\cd)$ replaced by $\bar{\mathbbm{u}}(\cd,\cd)$ introduced in \eqref{OC-C}.
 %
%
 In  Case (a), for  $(t,x)\in[0,T]\times\dbR^n$ such that    $W(t,x)= h(t,x)$, combined with the obstacle condition in reflected BSDE \eqref{RBSDE-u},
   we get
 $$W(t,x)= h(t,x)\ges Y^{t,x;\bar{\mathbbm{u}}}(t)=J\big(t,x;\bar{\mathbbm{u}}(\cd,X^{t,x;\bar{\mathbbm{u}}}(\cd))\big),\q (t,x)\in[0,T]\times\dbR^n.$$
 %

  In  Case (b), applying It\^o's formula to $W(\cd,X^{t,x;\bar{\mathbbm{u}}}(\cd))$ on $[t,T]$, we have
\small $$\ba{ll}
\ns\ds  W\big(s,X^{t,x;\bar{\mathbbm{u}}}(s)\big)   = \Phi\big ({{X^{t,x;\bar{\mathbbm{u}}}(T)}}\big) - \int_s^T W_x\big(r,X^{t,x;\bar{\mathbbm{u}}}(r)\big). \sigma\big (r,X^{t,x;\bar{\mathbbm{u}}}(r), \bar{\mathbbm{u}}(r,X^{t,x;\bar{\mathbbm{u}}}(r))\big) dB(r)\\
\ns\ds+ \int_s^T \Big\{f\big(r,X^{t,x;\bar{\mathbbm{u}}}(r),W\big(r,X^{t,x;\bar{\mathbbm{u}}}(r)\big),
W_x\big(r,X^{t,x;\bar{\mathbbm{u}}}(r)\big).\sigma\big (r,X^{t,x;\bar{\mathbbm{u}}}(r),\bar{\mathbbm{u}}(r,X^{t,x;\bar{\mathbbm{u}}}(r))\big),
\bar{\mathbbm{u}}(r,X^{t,x;\bar{\mathbbm{u}}}(r))\big)\\
\ns\ds - \frac{\partial }{{\partial r}}W\big(r,X^{t,x;\bar{\mathbbm{u}}}(r)\big)- \mathbb H\big(r,X^{t,x;\bar{\mathbbm{u}}}(r),W\big(r,X^{t,x;\bar{\mathbbm{u}}}(r)\big),
W_x\big(r,X^{t,x;\bar{\mathbbm{u}}}(r)\big),
 	 			W_{xx}\big(r,X^{t,x;\bar{\mathbbm{u}}}(r)\big),\bar{\mathbbm{u}}(r,X^{t,x;\bar{\mathbbm{u}}}(r))\big)\Big\} dr.
\ea$$
\normalsize
 Note that \eqref{OC-C}  and  Case (b) make ``$\les$" in \eqref{compar-f}  become ``$=$". Following the procedures in (i) and the uniqueness of the solution of BSDE, for all $n\in\dbN^*$, we get
 $$W\big(s,X^{t,x;\bar{\mathbbm{u}}}(s)\big) = {}^n Y^{t,x;\bar{\mathbbm{u}}}(s),\q s\in[t,T],\q \dbP\mbox{-a.s.}$$
Similarly to (i), 
 letting $n\to\i$  and $s=t$, we have
  	$W\big(t,x\big) = Y^{t,x;\bar{\mathbbm{u}}}(t) $.

Finally, combined with \eqref{Step1},  for any $(t,x)\in[0,T]\times\dbR^n$, we get
 $$W(t,x) =J\big(t,x;\bar{\mathbbm{u}}(\cd,X^{t,x;\bar{\mathbbm{u}}}(\cd)) \big)=\essinf_{u(\cd)\in\cU_{t,T}}J\big(t,x;u(\cd)\big)=V(t,x).$$
That is,  the classical solution $W(\cd,\cd)$ of HJB equation  \eqref{HJB} is indeed the value function $V(\cd,\cd)$ of Problem (RC)$_{t,x}$,
and
$\bar{\mathbbm{u}}(\cd,X^{t,x;\bar{\mathbbm{u}}}(\cd))$ is the optimal control of Problem (RC)$_{t,x}$.

 	 \end{proof}

\br\label{Re-H3} \sl
As we see in the proof of Theorem \ref{SVT-class},
the existence and uniqueness of the solutions of  the state equation \eqref{SDE-u} and reflected BSDE \eqref{RBSDE-u} under $\bar{\mathbbm{u}}(\cd,\cd)$  is necessary, so that the additional condition  $\bar{\mathbbm{u}} (\cd,\cd)\in\cL$ is necessary.
\er

%
%
%
%
%
%
%
%

\section{Stochastic Verification Theorem: Viscosity  Solutions}\label{Sec_SVP-C}

\par  In this section, we study the stochastic verification theorem of Problem (RC) within the framework of viscosity solutions.
As we know, there is no enough smoothness for the viscosity solutions so that we can not compute  their derivatives  directly  like the proof of Theorem \ref{SVT-class}. It is necessary to   introduce some new tools in this framework. The first are the  notions of second-order parabolic superdifferentials and subdifferentials(referring to \cite{YZ}).

%
\bde\sl  Let $(t,x)\in [0,T]\times \dbR^n$ and $w\in C([0,T]\times\dbR^n)$, the {second-order parabolic superdifferential} of $w$ at  $(t, x)$ is defined as
\bel{super}\ba{ll}
\ns\ds D^{1,2,+}_{t,x}w(t,x):=  \Big\{(q,p,P)\in \dbR\times \dbR^n\times \dbS^n\Big|\limsup\limits_{s\to t, y\to x}\frac 1{|s- t|+|y- x|^2}\big[w(s,y)-w(t,x)\\
\ns\ds\hskip 6.8cm -q(s-t)-\lan p,y-x\ran-\frac 12(y-x)^\top P(y-x) \big]\les 0\Big\},
\ea\ee
and the {second-order parabolic subdifferential} of $w$ at $(t,x)$ is defined as
\bel{sub}\ba{ll}
\ns\ds D^{1,2,-}_{t,x}w(t,x):=  \Big\{(q,p,P)\in \dbR\times \dbR^n\times \dbS^n\Big|\liminf\limits_{s\to t, y\to x}\frac 1{|s-t|+|y-x|^2}\big[w(s,y)-w(t,x)\\
\ns\ds\hskip6.46cm -q(s-t)-\lan p,y-x\ran-\frac 12(y-x)^\top P(y-x) \big]\ges 0\Big\}.
\ea\ee
\ede
The  second-order right parabolic superdifferential $D^{1,2,+}_{t+,x}w(t,x)$  and subdifferential $D^{1,2,-}_{t+,x}w(t,x)$  can also be defined by modifying $s\to t$ to $s\to t^+$ in \eqref{super} and \eqref{sub}, respectively.
\bl\label{L11}\sl Let  $w\in C([0,T]\times \mathbb{R}^n)$ and $(t_0,x_0)\in [0,T)\times \dbR^n$ be given. Then,
%
\no (i). $(q,p,P)\in D_{t+,x}^{1,2,+}w(t_0,x_0)$ if and only if there exists a function $\f\in  C^{1,2}([0,T]\times \mathbb{R}^n)$ such that, for any $ (t,x)\in [t_0,T]\times \dbR^n$, $(t,x)\neq (t_0,x_0)$,
$\f(t,x)>w(t,x),$  and $$\big(\f(t_0,x_0),\f_t(t_0,x_0),\f_x(t_0,x_0),\f_{xx}(t_0,x_0)\big)=\big(w(t_0,x_0),q,p,P\big).$$
\noindent {\rm(ii).} $(q,p,P)\in D_{t+,x}^{1,2,-}w(t_0,x_0)$ if and only if there exists a function $\f\in  C^{1,2}([0,T]\times \mathbb{R}^n)$ such that, for any $ (t,x)\in [t_0,T]\times \dbR^n$, $(t,x)\neq (t_0,x_0)$,
$\f(t,x)<w(t,x),$   and $$\big(\f(t_0,x_0),\f_t(t_0,x_0),\f_x(t_0,x_0),\f_{xx}(t_0,x_0)\big)=\big(w(t_0,x_0), q,p,P\big).$$
Moreover, if for some $k\ges 1$, $(t,x)\in [0,T]\times \dbR^n$,
 \bel{poly}|w(t,x)|\les C(1+|x|^k),\ee
then we can choose $\f$ such that $\f$, $\f_t$, $\f_x$, $\f_{xx}$  also satisfy \eqref{poly} with different constants $C$.

\el

The details of the above result  can be found in \cite{YZ, ZYL, GSZ1}.
It  will  provide us the smooth test functions to replace the viscosity solutions to compute the derivatives in the  proof of the verification theorem.

\ms

The following two results are borrowed from  \cite{YZ, GSZ1, CL-arxiv}.

\bl\label{lemma1} \sl Suppose {\bf(H1)}, let $(t,x)\in [0,T]\times \mathbb{R}^n$ be fixed and $ u (\cd) \in \cU_{t,T},$ $X^{t,x;u}(\cd)$  be the corresponding state process of \eqref{SDE}. By defining the following processes
$$
 \psi_1(r):=  b(r,X^{t,x;u}(r) ,u(r)  ),\q  \psi_2(r):=  \sigma\sigma^\top(r,X^{t,x;u}(r) ,u(r)  ),\q r\in[t,T],
 $$
we get  \be\lim\limits_{h\to0^+} \frac 1h \int_s^{s+h}\dbE|\psi_i(r)-\psi_i(s) |dr =0,\ \mbox{a.e.}, \ s\in[t,T],\ i=1,2.\ee

\el

\bl\label{le22}\sl
Let $g\in C([0,T])$ and extend $g$ to $(-\i,+\i)$ by setting $ g(t)=\left\{\ba{ll}g(0),\ t<0\\
g(t),\   t\in[0,T].\\
g(T),\ t>T\ea\right.$
  Suppose that for all $\d\in(0,T),$  there is a function  $\rho(\cd)\in L^1(0,T-\d;\mathbb{R})$ and some $h_0>0$, such that
\bel{}\frac{g(t+h)-g(t)}{h}\les \rho(t),\ \mbox{a.e.}, \ t\in [0,T-\d),\ h\les h_0,\ee
then,
$$g(\b)-g(\a)\les \int_\a^\b \limsup\limits_{h\to 0^+}\frac{g(t+h)-g(t)}{h}dt,\qq 0\les \a <\b\les T-\d. $$


\el
%
%
%

\subsection{The Main Theorem}

%
%
%
%
%
%
%
%
%
%
%
%

Before presenting the main result, we introduce the following two conditions.

\ss

\noindent {\bf (D1).} For all $x\in\dbR^n$ and $\d\in(0,T)$, for any $t,t'\in [0,T-\d]$, $ |\phi(t,x)-\phi(t',x)|\les C_{1,\d}(1+|x|)|t-t'| ,$ where $C_{1,\d}>0$ is a constant depending on $\d$;

\noindent {\bf (D2).} For all  $\d\in(0,T)$,  $\phi(t,\cd)$ is $C_{2,\d}$-semiconcave, uniformly in   $t\in [0,T-\d]$, i.e., there exists some constant  $C_{2,\d}>0$, such that $\phi(t,\cd)-C_{2,\d}|\cd|^2$ is concave on $\dbR^n$.

\ms

The following is the stochastic verification theorem within the framework of viscosity solutions.

\bt\label{SVT-VS}\sl
Let $W(\cd,\cd)\in C_p([0,T]\times\dbR^n)$ be the viscosity solution of   the HJB equation   with obstacle \eqref{HJB} and satisfy  the conditions  {\bf (D1)} and  {\bf (D2)}. For  $(t,x)\in [0,T]\times \dbR^n$, let $(\bar u(\cd),X^{t,x;\bar {u}}(\cd))$ be the admissible pair and  $(Y^{t,x; \bar{u}}(\cd),Z^{t,x; \bar{u}}(\cd),K^{t,x; \bar{u}}(\cd))$ solve reflected BSDE \eqref{RBSDE}  under the control process $\bar u(\cd)\in\cU_{t,T}$. Assume that  there exists  a triple of
$
\big(\bar q,\bar p, \bar P\big)\in L^2_\dbF(t,T;\dbR)\times L^2_\dbF(t,T;\dbR^n)\times L^2_\dbF(t,T;\dbS^n)
$
such that
$$\left\{\ba{ll}
\ns\ds\!\!\! {\rm(i).}\   \big(\bar q(s),\bar p(s) , \bar P(s) \big)\in D^{1,2,+}_{t+,x}W\big(s,X^{t,x; \bar{u}}(s) \big) , \ \mbox{a.e.}  \ s\in [t,T],\ \dbP\mbox{-a.s.};\\
\ns\ds\!\!\! {\rm(ii).}\     \bar p(s) .\si\big(s,X^{t,x; \bar{u}}(s) ,\bar{u}(s)\big)=Z^{t,x; \bar{u} }(s) ,\ \mbox{a.e.} \ s\in [t,T],\  \dbP\mbox{-a.s.};\\
\ns\ds\!\!\! {\rm(iii).} \  \dbE\int_t^T \[\bar q(s) + \dbH \big( s,X^{t,x; \bar{u} }(s),Y^{t,x; \bar{u}  }(s),\bar p(s) , \bar P(s)   , \bar{u}(s)  \big)\]ds \les0  .
\ea\right.$$
Then, $\bar u(\cd)$ is the optimal control of Problem (RC)$_{t,x}$.

\et

\begin{proof}
Firstly, from the uniqueness of the viscosity solution of  \eqref{HJB} (referring to  Lemma \ref{Le-vis}), we know, for any $(t,x)\in[0,T]\times\dbR^n$ and $u(\cd)\in\cU_{t,T}$
\bel{ee1}
W(t,x)=V(t,x)\les J\big(t,x;u(\cd)\big).
\ee

If we  fix some point $t_0\in [t,T]$ such that (i) holds true, and
 \bel{equ-lmn}\ba{ll}
 \ns\ds \lim\limits_{h\to 0^+} \frac 1h \int_{t_0}^{t_0+h} \dbE\big|b\big(r,X^{t,x; \bar{u}}(r) ,\bar u(r) \big)-b\big(t_0,X^{t,x; \bar{u} } (t_0),\bar u(t_0)\big)\big|dr=0,  \\
  \ns\ds \lim\limits_{h\to 0^+} \frac 1h \int_{t_0}^{t_0+h} \dbE\big|\si\si^\top\big(r,X^{t,x;\bar {u} }(r) ,\bar u(r)  \big)-\si\si^\top\big(t_0,X^{t,x; \bar {u}  } (t_0),\bar u(t_0)\big)\big|dr=0,
 \ea\ee
then, from condition (i) and Lemma \ref{lemma1}, we know the set of such points $t_0$ is of full measure in $[t,T]$.

Given $\cF_{t_0}^t:= \si\{B(r):t\les r\les t_0 \}$ augmented by all the $\dbP$-null sets in $\cF$, we fix $\o_0\in \O$ such that the regular conditional probability $\dbP(\cd\mid\cF_{t_0}^t)(\o_0)$  is well defined.
In this new probability space $\big(\O,\cF,\dbP(\cd\mid\cF_{t_0}^t)(\o_0)\big)$, the random variables $X^{t,x; \bar u }(t_0)$, $\bar q(t_0)$, $\bar p(t_0)$, $\bar P(t_0)$ are almost surely deterministic constants and equal to  $X^{t,x; \bar u} (t_0,\o_0)$, $\bar q(t_0,\o_0)$, $\bar p(t_0,\o_0)$, $\bar P(t_0,\o_0)$, respectively. Note that in this probability space the Brownian motion  $B$ is still a standard Brownian motion, although now $B(t_0)=B(t_0,\o_0)$ almost surely. Now the space is equipped with a new filtration $\{\cF_s^{t_0}\}_{t_0\les s\les T}$ and the control process $\bar {u}(\cd)$ is adapted to the new filtration.
 For $\o_0$, the process $X^{t,x; \bar u }(\cd) $ is a solution of \eqref{SDE} on $[t_0,T]$  in $\big(\O,\cF,\dbP(\cd\mid\cF^{t}_{t_0})(\o_0)\big)$ with the  initial condition $X^{t,x;  \bar u }(t_0)=X^{t,x; \bar u }(t_0,\o_0)$.

From (i) and   Lemma \ref{L11}-(i), we know there exists a function $\bar\f(\cd,\cd) \in  C^{1,2}([0,T]\times \mathbb{R}^n)$ such that, $W-\bar\f$ attains a strict maximum over $[t_0,T]\times\dbR^n$ at $\big(t_0,X^{t,x;\bar u }(t_0)\big)$, and
\bel{equ999}\ba{ll}
 \ns\ds \(\bar\f\big(t_0,X^{t,x;\bar u }(t_0) \big),\bar\f_t\big(t_0,X^{t,x;\bar u }(t_0)\big),\bar\f_x\big(t_0,X^{t,x;\bar u }(t_0)\big),\bar\f_{xx}\big(t_0,X^{t,x;\bar u }(t_0)\big)\)\\
 \ns\ds =\(W\big(t_0,X^{t,x;\bar {u} }(t_0,\o_0)\big),\bar q(t_0,\o_0), \bar p(t_0,\o_0),\bar P(t_0,\o_0)\).\ea\ee
The linear growth of $W(\cd,\cd)$ in Lemma \ref{Le-con} implies us
$\bar\f,$ $\bar\f_t,$ $ \bar\f_x,$ $\bar\f_{xx}$ are also linear growth in $x$, i.e.,
\bel{phi-lg}
|\bar\f(t,x)|+|\bar\f_t(t,x)|+|\bar\f_x(t,x)|+|\bar\f_{xx}(t,x)|\les C(1+|x|),\q (t,x)\in[0,T]\times\dbR^n.
\ee
 Note that, on the space $\big(\O,\cF,\dbP(\cd\mid\cF_{t_0}^t)(\o_0)\big)$, $\bar\f$ is the deterministic function when $(t_0,\o_0)$ is fixed.

 For any $h>0$, applying It\^o's formula to $\bar\f\big(\cd,X^{t,x; \bar u }(\cd)\big)$ on $[t_0,t_0+h]$, we  have
 $$
 \ba{ll}
 \ns\ds \bar \f\big(t_0+h,X^{t,x; \bar u }(t_0+h)\big)-\bar\f\big(t_0,X^{t,x;\bar u }(t_0)\big)\\
 \ns\ds=\int_{t_0}^{t_0+h}\[\bar\f_t\big(r,X^{t,x;\bar u }(r) \big)+\big\langle \bar\f_x\big(r,X^{t,x;\bar u }(r) \big), b\big(r,X^{t,x;\bar u }(r) , \bar{u}(r)\big)\big\rangle \\
 \ns\ds\qq\qq\q  + \tr\big(a  (r,X^{t,x;\bar u }(r) , \bar{u}(r) ) \bar\f_{xx} (r,X^{t,x;\bar u }(r)  )\big) \]dr\\
 \ns\ds\q +\int_{t_0}^{t_0+h}\big\langle  \bar\f_x\big(r,X^{t,x; \bar u }(r) \big), \si\big(r,X^{t,x;\bar u }(r) , \bar{u}(r)\big)\big\rangle dB(r).
 \ea
 $$
 %
 Note that the condition {\bf (H1)}, the regularity properties \eqref{est888}  of $X^{t,x; \bar u }(\cdot)$, as well as the choice  of $\bar\f$ satisfying \eqref{phi-lg}, imply us that all the   integrals in the above equality make sense.

 Taking the conditional expectation $\dbE_{t,t_0}[\cd]:=\dbE\big[\cd\mid\cF_{t_0}^t (\o_0) \big](\o_0) $, we get
 $$
 \ba{ll}
  \ns\ds \frac 1h\dbE_{t,t_0}\Big[W\big(t_0+h,X^{t,x;\bar u }(t_0+h)\big)-W\big(t_0,X^{t,x;\bar u }(t_0)\big) \Big]\\
 \ns\ds\les \frac 1h\dbE_{t,t_0}\Big[\bar\f\big(t_0+h,X^{t,x;\bar u }(t_0+h)\big)-\bar\f\big(t_0,X^{t,x;\bar u }(t_0)\big) \Big]\\
 \ns\ds= \frac 1h\dbE_{t,t_0}\Big[\int_{t_0}^{t_0+h}\Big(\bar\f_t\big(r,X^{t,x;\bar u }(r) \big)+\big\langle \bar\f_x\big(r,X^{t,x;\bar u }(r) \big), b\big(r,X^{t,x;\bar u }(r) , \bar{u}(r)\big)\big\rangle \\
 \ns\ds\qq\qq\qq\q\q\ + \tr\big(a (r,X^{t,x;\bar u }(r) ,\bar{u}(r) ) \bar\f_{xx} (r,X^{t,x;\bar u }(r) )\big) \Big)dr  \Big].
 \ea
 $$
By letting $h\to 0^+$, and applying  \eqref{equ-lmn}, \eqref{equ999},  we obtain
 \bel{ineq-111}
 \ba{ll}
  \ns\ds \frac 1h\limsup\limits_{h\to 0^+}\dbE_{t,t_0}\Big[W\big(t_0+h,X^{t,x; \bar u }(t_0+h)\big)-W\big(t_0,X^{t,x;\bar u }(t_0)\big) \Big]\\
 \ns\ds\les\lim\limits_{h\to0^+}\frac 1h\dbE_{t,t_0}\Big[\int_{t_0}^{t_0+h}\Big(\bar \f_t\big(r,X^{t,x;\bar u }(r) \big)+\big\langle \bar\f_x\big(r,X^{t,x;\bar u }(r) \big), b\big(r,X^{t,x;\bar u }(r) , \bar{u}(r)\big)\big\rangle \\
 \ns\ds\hskip4cm + \tr\big(a (r,X^{t,x;\bar u }(r) ,\bar{u}(r) ) \bar\f_{xx} (r,X^{t,x;\bar u }(r) )\big) \Big)dr  \Big]\\
 \ns\ds= \bar \f_t\big(t_0,X^{t,x;\bar u }(t_0)\big)+\big\langle \bar\f_x\big(t_0,X^{t,x;\bar u }(t_0)\big), b\big(t_0,X^{t,x;\bar u }(t_0), \bar u(t_0) \big)\big\rangle \\
 \ns\ds\q    + \tr\big(a (t_0,X^{t,x;\bar u }(t_0),\bar{u}(t_0) )\bar\f_{xx} (t_0,X^{t,x;\bar u }(t_0) )\big)   \\
 \ns\ds=\bar q(t_0)+\big\langle \bar p(t_0), b\big(t_0,X^{t,x;\bar u }(t_0),\bar{u}(t_0)\big)\big\rangle    + \tr\big(a  \big(t_0,X^{t,x;\bar u }(t_0),\bar{u}(t_0)\big)\bar P(t_0)\big).
 \ea
\ee
The above limit process (the first ``$=$") is similar to the one in \cite{GSZ1}, so we omit it to avoid repetition. 

Next, we claim that, for all $ \d\in(0,T)$ and the previous   $t_0$ lying in $[t,T-\d)$,
 for any  $h>0 $ with  $t_0+h\les T-\d$,
\bel{ineq-000}\ba{ll}
\ns\ds {\rm(a).}\ \   \frac 1h\dbE_{t,t_0}\[  W\big(t_0+h,X^{t,x; \bar u }(t_0+h)\big)-W\big(t_0,X^{t,x; \bar u }(t_0) \big)   \]\les   C(1 +|x|^2),\\
 \ns\ds {\rm(b).}\ \  \frac 1h \dbE\[   W\big(t_0+h,X^{t,x; \bar u }(t_0+h)\big)-W\big(t_0,X^{t,x; \bar u }(t_0) \big)  \]\les C(1 +|x|^2) .
 \ea\ee
In fact, from {\bf(D1)}, {\bf(D2)} and $\big(\bar q(t_0) ,\bar p(t_0) , \bar P(t_0) \big)\in D^{1,2,+}_{t+,x}W(t_0,X^{t,x; \bar u }(t_0))$, we know, for all $\d\in(0,T)$,  for any $h\in (0,T-t_0-\d]$,
\bel{} \ds W\big(t_0+h,X^{t,x; \bar u }(t_0+h)\big)-W\big(t_0,X^{t,x; \bar u }(t_0) \big)\les {\rm  I }+\rm{II}+\rm{ III},\ee
with
$$\left\{\ba{ll}
\ns\ds\!\!\! {\rm  I }:= C_{1,\d}\big(1+|X^{t,x; \bar u }(t_0+h)| \big)h,\qq C_{1,\d}  \mbox { is the one in {\bf(D1)} };\\
\ns\ds \!\!\! {\rm II}:=\big\langle \bar p(t_0) , X^{t,x; \bar u }(t_0+h)-X^{t,x; \bar u }(t_0) \big\rangle ,\\
\ns\ds\!\!\!  {\rm  III} :=C_{2,\d}|X^{t,x; \bar u }(t_0+h)-X^{t,x; \bar u }(t_0) |^2,\qq   C_{2,\d}    \mbox { comes from } {\bf(D2)}.
\ea\right.$$

Using  \eqref{phi-lg}, we have   $|\bar q(t_0) |+|\bar p(t_0) |+|\bar P(t_0) |\les C(1+|X^{t,x; \bar u }(t_0) | ).$
Further, combined with the estimate \eqref{est888}, we get
$$
 \ba{ll}
  \ns\ds \dbE _{t,t_0}\big[{\rm I }  \big]\les C_{1,\d} h+ C_{1,\d} h\Big(\dbE _{t,t_0}  \Big[\sup\limits_{ r\in [t_0,t_0+h]}|X^{t,x; \bar u }(r)|^2 \Big]\Big)^{\frac{1}{2}}\les  C h(1+ |x|^2 )^\frac12, \\
  %
  \ns\ds  \dbE _{t,t_0}\big[{\rm II} \big]
  %
  \les  \dbE_{t,t_0} \[ \big\langle \bar p(t_0),\int_{t_0}^{t_0+h}b\big(r,X^{t,x; \bar u }(r) ,\bar{u}(r)\big)dr\big\rangle  \]\\
\ns\ds\hskip1.3cm\les  \(\dbE_{t,t_0}\big[|\bar p(t_0)|^2  \big]\)^\frac12\(\dbE_{t,t_0}\[\big(\int_{t_0}^{t_0+h}b\big(r,X^{t,x; \bar u }(r) ,\bar{u}(r)\big)dr\big)^2   \] \)^\frac12\\
\ns\ds\hskip1.3cm\les Ch\(1+\dbE_{t,t_0}\[\sup\limits_{ r\in [t_0,t_0+h]}|X^{t,x; \bar u }(r)|^2  \]\) \les Ch(1+ |x|^2) ,
\ea$$
and
$$
 \ba{ll}
  \ns\ds  \dbE _{t,t_0}\big[{\rm III} \big]= C_{2,\d}\dbE _{t,t_0}\big[|X^{t,x; \bar u }(t_0+h)-X^{t,x; \bar u }(t_0) |^2  \big]  \\
\ns\ds\les 2C_{2,\d}\dbE _{t,t_0}\[ \(\int_{t_0}^{t_0+h}b\big(r,X^{t,x; \bar u }(r) ,\bar{u}(r)\big)dr  \)^2\]+ 2C_{2,\d}\dbE _{t,t_0} \[ \(\int_{t_0}^{t_0+h}\si\big(r,X^{t,x; \bar u }(r) ,\bar{u}(r)\big)dB(r)   \)^2\]\\
\ns\ds\les Ch \(1+\dbE _{t,t_0}\[\sup\limits_{ r\in [t_0,t_0+h]}|X^{t,x; \bar u }(r)|^2  \]\) \les Ch(1+ |x|^2) .
\ea$$
Therefore,
$$
 \ba{ll}
 \ns\ds \frac 1h\dbE_{t,t_0}\[ W\big(t_0+h,X^{t,x; \bar u }(t_0+h)\big)-W\big(t_0,X^{t,x; \bar u }(t_0) \big) \]  \les \frac 1h\dbE_{t,t_0} \[ {{\rm  I }+\rm{II}+\rm{ III}}\] \les C(1+ |x|^2) .
\ea$$
All the above constants $C$ can be different and do not depend on $t_0$.
Further, by taking the expectation on the both sides
of the above inequality, we get \eqref{ineq-000}-(b).

Taking expectation on the both sides of \eqref{ineq-111}, and applying Fatou's Lemma (needing \eqref{ineq-000}-(a)), we have
 $$
 \ba{ll}
  \ns\ds \limsup\limits_{h\to 0^+}\frac 1h\dbE\Big[W\big(t_0+h,X^{t,x;\bar u }(t_0+h)\big)-W\big(t_0,X^{t,x;\bar u }(t_0)\big) \Big]\\
  \ns\ds= \limsup\limits_{h\to 0^+}\frac 1h\dbE\[\dbE_{t,t_0}\big[W\big(t_0+h,X^{t,x;\bar u }(t_0+h)\big)-W\big(t_0,X^{t,x;\bar u }(t_0)\big) \big]\]\\
 \ns\ds\les \dbE\[\limsup\limits_{h\to 0^+}\frac 1h\dbE_{t,t_0}\big[W\big(t_0+h,X^{t,x;\bar u }(t_0+h)\big)-W\big(t_0,X^{t,x;\bar u }(t_0)\big) \big]\]\\
 \ns\ds\les  \bar q(t_0) +\big\langle \bar p(t_0), b\big(t_0,X^{t,x;\bar u }(t_0),\bar{u}(t_0)\big)\big\rangle  + \tr\big(a (t_0,X^{t,x;\bar u }(t_0),\bar{u}(t_0))\bar P(t_0)\big).
 %
 \ea
 $$
 %
 %
Due to the set of such points $t_0$ being of full measure in $[t,T-\d]$, by  applying Lemma \ref{le22} (needing \eqref{ineq-000}-(b)), for any  $u(\cd) \in \cU_{t,T},$ we have
  $$
 \ba{ll}
  \ns\ds  \dbE\big[W\big(T-\d,X^{t,x;\bar u }(T-\d) \big)-W(t,x) \big] \\
 \ns\ds\les \int_t^{T-\d} \dbE\[\bar q(s) +\big\langle \bar p(s) , b\big(s,X^{t,x;\bar u }(s) ,\bar{u}(s)\big)\big\rangle +\tr\big(a (s,X^{t,x;\bar u }(s) ,\bar{u}(s))\bar P(s)\big)  \]ds .\\
 %
 \ea$$
According to Lebesgue  dominated convergence theorem apnd  the continuity properties of $W(\cd,\cd)$ and $X^{t,x;\bar u }(\cd)$, letting $\d\to0$ in the above, we get
  $$
 \ba{ll}
  \ns\ds  \dbE\big[W\big(T,X^{t,x;\bar u }(T) \big)-W(t,x) \big]= \dbE\big[\Phi\big(X^{t,x;\bar u }(T) \big)-W(t,x) \big]\\
 \ns\ds\les \int_t^T \dbE\[\bar q(s) +\big\langle \bar p(s) , b\big(s,X^{t,x;\bar u }(s) ,\bar{u}(s)\big)\big\rangle +\tr\big(a (s,X^{t,x;\bar u }(s) ,\bar{u}(s))\bar P(s)\big)  \]ds \\
 \ns\ds\les - \dbE \[\int_t^T f\big(s,X^{t,x;\bar u }(s) ,Y^{t,x;\bar u }(s)  ,\bar p(s) .\sigma(s, X^{t,x;\bar u }(s) ,\bar{u}(s)),
 \bar{u}(s)\big) ds\]\\
  \ns\ds\les - \dbE\[\int_t^T f\big(s,X^{t,x;\bar u }(s) ,Y^{t,x;\bar u }(s)  ,Z^{t,x;\bar u }(s) ,
 \bar{u}(s)\big) ds -K^{t,x;\bar u}(T)\],
 \ea$$
  where we have used the conditions (ii), (iii) and $K^{t,x; \bar{u}}(\cd)\in \cA^2_c(t,T;\dbR^n)$.
That is, for any $(t,x)\in[0,T]\times\dbR^n$,
 $$
 \ba{ll}
  \ns\ds W(t,x)\ges  \dbE\Big[\Phi\big(X^{t,x;\bar u }(T) \big)   + \int_t^T  f\big(s,X^{t,x;\bar u }(s) ,Y^{t,x;\bar u }(s)  ,Z^{t,x;\bar u }(s) , \bar{u}(s)\big) ds-K^{t,x;\bar u}(T)\Big]=J\big(t,x; \bar u(\cd)\big). \\
 \ea$$
 Combined with \eqref{ee1}, we get
 $
 W(t,x)=  V(t,x)=J\big(t,x; \bar u (\cd)\big),
$
which means $\bar u(\cd)$ is the optimal control of Problem (RC)$_{t,x}$.

\end{proof}

Note that, in order to obtain the stochastic verification theorem of Problem (RC), the viscosity solution of \eqref{HJB} needs to satisfy the additional conditions {\bf(D1)} and {\bf(D2)}.
We want to say it is possible, though   a bit  restrictive. Now, we present the   conditions ensuring  the viscosity solution of \eqref{HJB} to satisfy {\bf(D1)} and {\bf(D2)} as follows.

\ss

\no {\bf (A1).} (i) The functions $b(\cd,x,u)$, $\si(\cd,x,u)$, $f(\cd,x,y,z,u)$, $h(\cd,x)$   are Lipschitz in $t\in[0,T]$, uniformly with respect to $(x,u)\in\dbR^n\times U$;

\hskip0.6cm (ii) The functions $b,$ $\si$, $f$, $\F$, $h$ are bounded.

\no {\bf (A2).} (i) $f(t,x,y,z,u)$ is semiconcave in $(x,y,z)\in\dbR^n\times \dbR\times\dbR^d,$ uniformly with respect to $(t,u)\in [0,T]\times U$; $\F(x)$ is semiconcave in $x\in\dbR^n$;

\hskip0.6cm (ii) $b(t,x,u)$ and $\si(t,x,u)$ are differentiable in $x\in\dbR^n$, and the corresponding  first order partial derivatives are continuous in $(t,x,u)$, Lipschitz continuous in $x$, uniformly with respect to $(t,u)\in[0,T]\times  U$.

\no {\bf (A3).}   $f(t,x,y,z,u)=f(t,x,y,u)$ is independent of $z$, and  $h$ is semiconcave in $x\in \dbR^n$.

\no {\bf (A4).} $h(t,x)=h\in \dbR$, $(t,x)\in[0,T]\times \dbR^n$.

\ms

From Theorem 1.1, Remark 1.3 and Theorem 3.1 in \cite{BHL2012}, we get the following two results.

\bl\sl
Under  {\bf (H1)}, {\bf (H2)} and {\bf (A1)}, the value function $V(\cd,\cd)$ defined  by \eqref{OPC} is joint Lipschitz continuous in $(t,x)\in [0,T-\d]\times \dbR^n$ for all $\d\in(0,T)$, i.e., there exists $C_{1,\d}>0$ such that
for any $(t,x)$, $(t',x')\in [0,T-\d]\times \dbR^n $,
  $$|V(t,x)-V(t',x')|\les C_{1,\d}(|t-t'|+|x-x'|).$$

\el

\bl\sl Suppose that  {\bf (H1)}, {\bf (H2)} and {\bf (A1)}, {\bf (A2)} hold, as well as {\bf (A3)} or  {\bf (A4)}. Then, for all $\d\in(0,T)$, there exists some $C_{2,\d}>0$ such that the value function $V(\cd,\cd)$ defined  by \eqref{OPC} is $C_{2,\d}$-semiconcave, i.e.,
$V(t,\cd)-C_{2,\d}|\cd|^2$ is concave, uniformly in $t\in[0,T]$.

\el

Combining Lemma \ref{Le-vis}  with the above two lemmas, we know  {\bf(D1)}, {\bf(D2)} can be satisfied when $W(\cd,\cd)$ is taken as the value function $V(\cd,\cd)$ of Problem (RC)$_{t,x}$  under the conditions {\bf (H1)}, {\bf (H2)}, {\bf (A1)}, {\bf (A2)},  as well as {\bf (A3)} or  {\bf (A4)}.  Therefore, the assumptions  {\bf(D1)} and {\bf(D2)} in Theorem \ref{SVT-VS} are acceptable.

\subsection{Feedback Optimal Control }
\par  In this subsection, we shall construct the feedback optimal control of Problem (RC) from the viscosity solution of HJB equation with obstacle \eqref{HJB}.

\bl\label{th4.10}\sl Assume {\bf (H1)} and {\bf (H2)}. Then the value function $V(\cd,\cd)$ in \eqref{OPC} is the only function satisfying Lemma \ref{Le-con}  and   the following: for any $(t,x)\in [0,T]\times \dbR^n$,
\bel{VIS-inequ}
\left\{\ba{ll}
\ns\ds\!\!\!  \max\big\{V(t,x)- h(t,x),-q-\inf_{u\in U}\dbH(t,x, V(t,x),p,P,u) \big\}\les 0,\q\forall (q,p,P)\in D^{1,2,+}_{t+,x}V(t,x);\\
\ns\ds\!\!\!  \max\big\{V(t,x)- h(t,x),-q-\inf_{u\in U}\dbH(t,x,V(t,x),p,P,u)\big\}\ges 0,\q \forall (q,p,P)\in D^{1,2,-}_{t+,x}V(t,x);\\
\ns\ds\!\!\! V(T,x)=\Phi(x).
\ea\right.
\ee

\el

\begin{proof}
	From Lemma \ref{Le-vis}, we know the value function $V(\cd,\cd)\in C_p([0,T]\times\dbR^n)$ is the  unique viscosity solution of \eqref{HJB}.  By Lemma \ref{L11}-(i), for any $(q,p,P)\in D^{1,2,+}_{t+,x}V(t,x)$,  we can find a  function $\f\in  C^{1,2}([0,T]\times \mathbb{R}^n)$ such that, for any $ (s,y)\in [t,T]\times \dbR^n$, $(s,y)\neq (t,x)$,
	$\f(s,y)>V\big(s,y),$  and $$\big(\f(t,x),\f_t(t,x),\f_x(t,x),\f_{xx}(t,x)\big)=\big(V(t,x),q,p,P\big).$$	
	Then, by following the procedures   in \cite{WY-2008, BL-2011} (only the right limit in time will be used there),  for the above test function  $\f$, we have
	$$\max\big\{V(t,x)- h(t,x),-\f_t(t,x)-\inf_{u\in U}\dbH(t,x, V(t,x),\f_x(t,x),\f_{xx}(t,x),u) \big\}\les 0,$$
	which results in the first inequality in \eqref{VIS-inequ}.
	The details of the second one in \eqref{VIS-inequ} is similar.

	Further, the uniqueness comes from the uniqueness of the viscosity solution of  \eqref{HJB} and $D^{1,2,+}_{t,x}V(t,x)\subseteq D^{1,2,+}_{t+,x}V(t,x)$, $D^{1,2,-}_{t,x}V(t,x)\subseteq D^{1,2,-}_{t+,x}V(t,x)$.

\end{proof}

\bt\label{th4.11} \sl Assume {\bf (H1)}-{\bf (H3)}. Suppose $W\in   C_p([0,T]\times\dbR^n)$ satisfying {\bf(D1), \bf(D2)} is the viscosity solution of
\eqref{HJB}. Then, for each $(t,x)\in [0,T]\times\dbR^n$,  
\bel{equ012}
\ds\inf_{(q,p,P,u)\in D^{1,2,+}_{t+,x}W(t,x)\times U}\[q+\dbH(t,x, W(t,x),p,P,u)\]\ges W(t,x)-h(t,x).
\ee
Further,  if $\mathbbm{u}(\cdot,\cdot)\in\mathcal{L}$ and for all $(t,x)\in[0,T]\times \dbR^n$, $ \mathbbm{q},$ $\mathbbm{p},$ $\mathbbm{P}$ are measurable functions
satisfying $(\mathbbm{q},\mathbbm{p},\mathbbm{P})\in D^{1,2,+}_{t+,x}W(t,x)$, and
\bel{equ015}\left\{\ba{ll}
\ns\ds\!\!\! {\rm (i).} \ \dbE\int_t^T\[\mathbbm{q}(s,X^{t,x;\mathbbm{u}}(s)) +\dbH\big(s,X^{t,x;\mathbbm{u}}(s) ,Y^{t,x;\mathbbm{u} }(s),  \Theta(s,X^{t,x;\mathbbm{u}}(s)),\mathbbm{u}(s,X^{t,x;\mathbbm{u}}(s)) \big)\]ds \les 0,\\
\ns\ds\!\!\! {\rm (ii).} \     \mathbbm{p}(s,X^{t,x;\mathbbm{u}}(s)).\si\big(s,X^{t,x;\mathbbm{u} }(s) ,\mathbbm{u}(s,X^{t,x;\mathbbm{u}}(s))  \big)=Z^{t,x;\mathbbm{u}  }(s), \mbox{ a.e.}.\ s\in[t,T], \mbox{ P-a.s.},
\ea\right.\ee
where $ X^{t,x;\mathbbm{u} }(\cd)$, $(Y^{t,x;\mathbbm{u} }(\cd),Z^{t,x;\mathbbm{u} }(\cd),K^{t,x;\mathbbm{u} }(\cd))$ satisfy   \eqref{SDE-u} and  \eqref{RBSDE-u} with $ \mathbbm{u} (\cd,X^{t,x;\mathbbm{u} }(\cd)) \in \cU_{t,T} $, respectively, and $\Th (\cd,X^{t,x;\mathbbm{u}}(\cd))=\big (\mathbbm{p}(\cd,X^{t,x;\mathbbm{u}}(\cd)),\mathbbm{P}(\cd,X^{t,x;\mathbbm{u}}(\cd)) \big).$
Then,     $ \mathbbm{u}(\cd,\cd)$ is an optimal feedback  control law   of Problem (RC)$_{t,x}$.

\et

\begin{proof} \emph{Step 1:}
	From the uniqueness of the viscosity solution of HJB equation \eqref{HJB} and   Lemma \ref{th4.10},
	for any $(t,x)\in [0,T]\times \dbR^n$,   $ (q,p,P)\in D^{1,2,+}_{t+,x}W(t,x)$, we have
	$$  W(t,x)\les  h(t,x) \mbox{ and }  -q-\inf_{u\in U}\dbH(t,x, W(t,x),p,P,u)  \les 0,
	$$
	Then, for any  $u\in U  $,
	$$
	q+ \dbH(t,x, W(t,x),p,P,u) \ges  q+\inf_{u\in U}\dbH(t,x, W(t,x),p,P,u)  \ges W(t,x)-h(t,x),
	$$
	i.e. \eqref{equ012} holds true.
	
	\emph{Step 2:} For any $(t,x)\in [0,T]\times \dbR^n$, the admissble feedback control law  $\mathbbm{u}(\cd,\cd)$ and the solution $X^{t,x;\mathbbm{u} }(\cd)$  of \eqref{SDE-u}, we set
	$$\ba{ll}
	\bar u(s):= \mathbbm{u}(s,X^{t,x;\mathbbm{u} }(s) ), \enspace \bar q(s):= \mathbbm{q}(s,X^{t,x;\mathbbm{u} }(s) ),\enspace \bar p(s):= \mathbbm{p}(s,X^{t,x;\mathbbm{u} }(s) ),\enspace  \bar P(s) := \mathbbm{P}(s,X^{t,x;u\mathbbm{u}}_s ),\enspace s\in[t,T].
	\ea$$
	By \eqref{equ015},  $\bar u(\cd)$, $X^{t,x;\mathbbm{u} }(\cd)$,  $(Y^{t,x;\mathbbm{u} }(\cd),Z^{t,x;\mathbbm{u} }(\cd),K^{t,x;\mathbbm{u} }(\cd))$ and  the above $(\bar q,\bar p, \bar P)$ satisfy (i), (ii) and (iii) in Theorem \ref{SVT-VS}, so  $ \bar u(\cd)$ is the optimal control, i.e.,  $\mathbbm{u}(\cd,\cd)$ is the optimal feedback control law.
	
\end{proof}

Finally, we have a look at the procedures of finding the optimal feedback control law.
By Theorem \ref{th4.11}, we  can get the candidate of  optimal feedback control law by minimizing $$q+ \dbH(t,x, W(t,x),p,P,u)$$
over $D^{1,2,+}_{t+,x}W(t,x)\times U$ such that \eqref{equ012} holds true.
Further, to ensure the candidate to be the true optimal feedback control law, there are three  things to do. The first is  to    obtain the measurable selection $\mathbbm{u}(\cd,\cd)$  and
$(\mathbbm{q},\mathbbm{p},\mathbbm{P})$ of $ D^{1,2,+}_{t+,x}W(t,x)$. Secondly,  we need to make sure
the candidate $\mathbbm{u}(\cd,\cd)$ is admissible, that is,  the existence of the solutions of  SDE \eqref{SDE-u} and RBSDE \eqref{RBSDE-u}.
Finally,   \eqref{equ015}  is still waiting for validation.   Especially, \eqref{equ015}-(ii) is necessary for  the control problems involving BSDEs.

In the above three steps, the first  and
the relevant  discussion on  SDE \eqref{SDE-u} in the second  step  have been made in Section 6, Chapter 5 of \cite{YZ}.
The remaining  is still blank and will be a big project. We skip it now and  hope to get some research results    in the future works.

 \section{Example}
\par In this section, we give two specific examples to illustrate that the obtained verification theorems give a way to construct an optimal control or to test whether a given admissible control is optimal. For simplicity, we assume $n=d=1$ in this section.
 The first example  is within the framework of  the classical solution.

\begin{example}\label{CEP}\sl
For any initial pair $(t,x)\in[0,T]\times\dbR$,   consider the following controlled system,
\bel{CEP-SDE}
 \left\{
  \ba{ll}
  \ns\ds\!\!\! dX(s)=(X(s)+u(s))ds+X(s)dB(s),\q s\in [t,T],\\
  \ns\ds\!\!\! X(t)=x.
  \ea
 \right.
\ee
with the control process $u(\cd)$ valued  in $U=[0,1]$. For   $(t,x)\in[0,T]\times\dbR^n$, by selecting $u(\cd)$, we shall minimize the following cost functional
$$J(t,x;u(\cd)):=Y(t),$$
where  $(Y(\cd),Z(\cd),K(\cd))$ is the solution of the following  BSDE with upper obstacle,
\bel{CEP-RBSDE}
 \left\{
  \ba{ll}
  \ns\ds\!\!\! Y(s)=X(T)+\int_s^T( Y(r)+u(r))dr-(K(T)-K(s))-\int_s^T Z(r)dB(r),\q s\in [t,T],\\
  \ns\ds\!\!\! Y(s)\les X(s)e^{2T},\q \ae s\in [t,T],\\
  \ns\ds\!\!\! \int_t^T \big( X(r)e^{2T}-Y(r) \big) dK(r)=0.\\
  \ea
 \right.
\ee
By the previous preliminaries, the above  control problem  which we denote by Problem $(C)_{t,x}$ makes sense obviously.
%
%
%
%

%
%
%
   Now we convert to  the  obstacle problem of HJB equation associated with Problem $(C)_{t,x}$ as follows,
\bel{CEP-HJB}
 \left\{
  \ba{ll}
  \ns\ds\!\!\!  \max\Big\{\! W(t,x)- x e^{2T}, -\frac{\partial}{\partial t}W(t,x)\!-\!\inf_{u\in [0,1]} \( \frac{1}{2}x^2W_{xx}(t,x)+(x+u)W_x(t,x)+W(t,x)+u \) \!\Big\}=0,\\
  \ns\ds\hskip11.8cm \ (t,x)\in[0,T]\times\dbR,\\
  \ns\ds\!\!\! W(T,x)=\Phi(x), \q x\in \dbR.
  \ea
 \right.
\ee
It is not difficult to verify  directly that the function $W(t,x)=xe^{2T-2t}\in C^{1,2}([0,T]\times\dbR)$ is the classical solution of \eqref{CEP-HJB}.
Therefore,
applying our first main result {\rm(Theorem \ref{SVT-class})}, from \eqref{OC-C},  $\bar u(\cd)\equiv 0$ is the  optimal control. In this case, the optimal trajectory   of Problem $(C)_{t,x}$  is
$$\bar X(s)=X^{t,x;\bar u}(s)=x\exp\Big\{\frac12(s-t)+B(s)-B(t)\Big\} ,\q s\in[t,T].$$
 For $x_0\in\dbR$, the optimal pair of Problem $(C)_{0,x_0}$ is $\displaystyle \(0,x_0\exp\Big\{\frac12s+B(s)\Big\}\)$, $s\in[0,T]$.
\end{example}

Next let's look at  the case with the viscosity solution.

\begin{example}\label{VEP}\sl
Given the control domain $U=[1,2]$.  For any   initial pair $(t,x)\in[0,T]\times\dbR^n$,  consider the following controlled system,
\bel{VEP-SDE}
 \left\{
  \ba{ll}
  \ns\ds\!\!\!dX(s)=X(s)u(s)ds+X(s)dB(s),\q s\in [t,T],\\
  \ns\ds\!\!\! X(t)=x,
  \ea
 \right.
\ee
and the following   RBSDE,
\bel{VEP-RBSDE}
 \left\{
  \ba{ll}
  \ns\ds\!\!\! Y(s)=X(T)-\int_t^T |Y(r)|dr-(K(T)-K(s))+\int_t^T Z(r)dB(r),\q s\in [t,T],\\
  \ns\ds\!\!\! Y(s)\les h(s,X(s)),\q \ae s\in [t,T],\\
  \ns\ds\!\!\! \int_t^T \big( h(r,X(r))-Y(r) \big) dK(r)=0,
  \ea
 \right.
\ee
where the obstacle function
$h(t,x)=\left\{
         \ba{ll}
         \ns\ds\!\!\! x,\q (t,x)\in[0,T]\times\dbR^+,\\
         \ns\ds\!\!\! 0,\q (t,x)\in[0,T]\times (\dbR^-\cup\{0\}).\\
         \ea
         \right.$
 Clearly,  \eqref{VEP-SDE} and \eqref{VEP-RBSDE} are well-posed. Therefore,  for $(t,x)\in[0,T]\times\dbR$, we define the cost functional  as follows,
$$J(t,x;u(\cd)):=Y(t).$$
Our control problem (denoted by Problem $(\dbC)_{t,x}$) is to minimize $ J(t,x;u(\cd))$ by selecting $u(\cd)\in U$, and its value function is $V(\cd,\cd) .$

Consider the following   obstacle problem of HJB  equation,  
\bel{VEP-HJB}
 \left\{
  \ba{ll}
  \ns\ds\!\!\!  \max\Big\{ W(t,x)- h(t,x), -\frac{\partial}{\partial t}W(t,x)-\inf_{u\in [1,2]} \( \frac{1}{2}x^2W_{xx}(t,x)+ xu W_x(t,x)-|W(t,x)| \) \Big\}=0,\\
  \ns\ds\hskip11.1cm \ (t,x)\in[0,T]\times\dbR,\\
  \ns\ds\!\!\! W(T,x)=\Phi(x), \q x\in \dbR.
  \ea
 \right.
\ee
Obviously, the following
\bel{Ex-vis}W(t,x)=\left\{
         \ba{ll}
         \ns\ds\!\!\! x,\hskip1.2cm \  (t,x)\in[0,T]\times\dbR^+,\\
         \ns\ds\!\!\! xe^{3T-3t},\q (t,x)\in[0,T]\times (\dbR^-\cup\{0\}),\\
         \ea
         \right.\ee
is not differentiable at $(t,0)$, for any $t\in[0,T]$.
Using the definition of viscosity solution (referring to \cite{BL-2011, WY-2008}),  we can check $W(\cd,\cd)$ in \eqref{Ex-vis} is  indeed a viscosity solution of \eqref{VEP-HJB}. Moreover, $W(\cd,\cd)\in C_p([0,T]\times \dbR)$ and   satisfies the conditions {\rm (D1)} and {\rm (D2)}.
%

%
Let us consider   an admissible control $\bar u(\cd)\equiv 1$ for the initial pair $(t,x)=(0,0)$. The corresponding trajectory  $\bar X(\cd)=X^{0,0;\bar u}(\cd)\equiv 0$. In this case,
$$ D^{1,2,+}_{t+,x}W\big(s, \bar X(s) \big)=[0,+\infty)\times[1,e^{3T-3s}]\times [0,+\infty),\q s\in[0,T].$$
By taking
$\big(\bar q(s), \bar p(s), \bar P(s) \big)=(0,1,0)\in D^{1,2,+}_{t+,x}W\big(s, \bar X(s) \big), $ $s\in[0,T],$
 it is easy to check that
$$\left\{\begin{aligned}
& \bar p(s) \cd\bar X(s)=0=\bar Z(s) , \ s\in [0,T];\\
&
  \dbE \int_t^T \( \bar q(s) + \frac{1}{2} \big( \bar X(s) \big) ^2 \bar P(s)+ \bar X(s) \bar p(s)-|\bar Y(s)| \)ds\les 0, \ s\in [0,T],\end{aligned}\right.$$
  which is in fact the conditions  (ii) and (iii) in Theorem \ref{SVT-VS}.
Note that in the above,  under $\bar u(\cd)$,  the solution $(\bar Y(\cd),\bar Z(\cd),\bar K(\cd))$ of  \eqref{VEP-RBSDE}  is $(0,0,0)$. Therefore, by   Theorem \ref{SVT-VS},
we get  $(\bar X(\cd),\bar u(\cd))$ is indeed the  optimal pair of Problem $(\dbC)_{0,0}$.

\end{example}

\end{document}